\documentclass[10pt,twoside,a4paper,intlimits]{amsart}

\usepackage{a4}
\usepackage[latin1]{inputenc}

\usepackage[centertags]{amsmath}
\usepackage{amsthm}
\usepackage{amssymb}
\usepackage{amsfonts}
\usepackage[mathscr]{eucal}
\usepackage{stmaryrd}
\usepackage{bbm} 
\usepackage{enumitem}
\usepackage{diplomcommands}

\usepackage{graphicx}

\usepackage{color}
\definecolor{lightgrey}{rgb}{.7,.7,.7}

\usepackage[frak,color,hyperref]{paper_diening}

\newtheorem{theorem}{Theorem}
\newtheorem{corollary}[theorem]{Corollary}
\newtheorem{lemma}[theorem]{Lemma}

\newtheorem{definition}[theorem]{Definition}
\newtheorem{remark}[theorem]{Remark}
\theoremstyle{definition}

\numberwithin{equation}{section}

\newcounter{notectr}
\newcommand{\note}[1]{\ifthenelse{\thenotectr=1}{#1}{}}

\newlength{\boxparameter}
\newcommand{\setRR}{\setR^n}

\newcommand{\pp}{p(\cdot)}
\newcommand{\pps}{p'(\cdot)}

\newcommand{\wh}[1]{\widehat{#1}}
\newcommand{\pa}{\partial}

\newcommand{\wf}{\widehat{f}}

\newcommand{\wv}{\widehat{V}}

\newcommand{\lpo}{L^{\pp}(\Omega)}
\newcommand{\wepo}{W^{1,\pp}(\Omega)}
\newcommand{\wzpo}{W^{2,\pp}(\Omega)}

\newcommand{\tfo}{C_0^\infty(\Omega)}
\newcommand{\lpr}{L^{\pp}(\setRR)}

\newcommand{\tfr}{C_0^\infty(\setRR)}

\newcommand{\normpo}[1]{\normtmp{}{#1}_{L^{\pp}(\Omega)}}
\newcommand{\normepo}[1]{\normtmp{}{#1}_{W^{1,\pp}(\Omega)}}
\newcommand{\normzpo}[1]{\normtmp{}{#1}_{W^{2,\pp}(\Omega)}}
\newcommand{\normpr}[1]{\normtmp{}{#1}_{L^{\pp}(\setRR)}}

\newcommand{\normnegpo}[1]{\normtmp{}{#1}_{W^{-1,\pp}(\Omega)}}

\newcommand{\beq}{\begin{equation*}}
\newcommand{\eeq}{\end{equation*}}
\newcommand{\bea}{\begin{eqnarray}}
\newcommand{\eea}{\end{eqnarray}}
\newcommand{\ben}{\begin{enumerate}}
\newcommand{\een}{\end{enumerate}}
\newcommand{\bal}{\begin{aligned}}
\newcommand{\eal}{\end{aligned}}
\newcommand{\PPln}{\mathcal{P}^{\log}}

\begin{document}


\title{The Stokes and Poisson problem in variable exponent spaces}

\author{L. Diening, D. Lengeler and M. \Ruzicka{}}

\address{Mathematisches Institut, Eckerstr. 1, 79104 Freiburg i. Br.,
  Germany.}%
\thanks{}%
\email{}

\begin{abstract}
  We study the Stokes and Poisson problem in the context of variable exponent spaces. We prove existence of strong and weak solutions for bounded domains with $C^{1,1}$ boundary with inhomogenous boundary values. The result is based on generalizations of the classical theories of \Calderon-Zygmund and Agmon-Douglis-Nirenberg to variable exponent spaces.
\end{abstract}



\maketitle

\section{Introduction}
\label{sec:introduction}
In the last decades, the generalized Lebesgue spaces $L^{p(\cdot)}$ and the corresponding generalized Sobolev spaces $W^{k,p(\cdot)}$ have attracted more and more attention. Before 1990 pioneering work has been done by Orlicz, Nakano, Hudzik, Musielak, and other authors. One of the first who studied problems with variable exponents in the context of variational integrals was Zhikov in his pioneering paper \cite{Zhi86} and subsequent works including \cite{Zhi95}, \cite{Zhi04}, \cite{Zhi08}. In the last twenty years, many new works have been devoted to the study of variable exponent spaces. We refer to \Kovacik, \Rakosnik{} \cite{KovR91}, Samko \cite{Sam98}, \cite{Sam99}, Fan, Zhao \cite{FanZ01}, Cruz-Uribe, Fiorenza, Martell, P{\'e}rez \cite{CruFMP06}, Diening~\cite{Die04}, \cite{Die04riesz}, Diening, \Ruzicka\ \cite{DieR03}, Diening, Harjulehto, H{\"a}sto, \Ruzicka{} \cite{DieHHR10} for properties of these spaces such as reflexivity, denseness of smooth functions, and Sobolev type embeddings, and for the treatment of operators of harmonic analysis in the variable exponent context. The study of these spaces has been stimulated by problems in elasticity, fluid dynamics, calculus of variations, and differential equations with $p(x)$-growth conditions. For example, in \Ruzicka, Rajagopal \cite{RajR96} one can find a model of electrorheological fluids, where the essential part of the energy is given by $\int |\boldmath Df(x)|^{p(x)}dx$, where $\boldmath Df(x)$ is the symmetric part of the gradient $\boldmath \nabla f$. The same type of energy also appears in a model proposed by Zhikov \cite{Zhi08} for the thermistor problem. This energy also appears in the investigations of variational integrals with non-standard growth, see e.g. Zhikov \cite{Zhi86}, Marcellini \cite{Mar91}, Acerbi, Mingione \cite{AceM01}.

Regularity results for the Stokes system and the Poisson equation belong to the most classical problems treated in the theory of partial differential equations and fluid dynamics and often occur as auxiliary problems in the treatment of nonlinear equations. In this paper we generalize some of these results to the variable exponent context. Besides being of interest in their own as generalizations of classical results to interesting new function spaces, these results are of great importance in the analysis of the nonlinear equations occurring in the study of the fluid mechanical problems mentioned above. Of course, the whole treatment applies to a much larger class of elliptic problems. 

We develop the analysis of the Stokes system in depth, while the results on the Poisson equation will be stated without proofs. For a sketch of the proofs we refer the reader to \cite{DieHHR10}, for full details on both problems see \cite{Len08}. In fact, the treatment of the Poisson equation is much simpler than that of the Stokes system and general elliptic problems. This is due to a symmetry of the fundamental solution of the Laplacian in the half-space by which the regularity near the boundary may be established without the use of the Agmon-Douglis-Nirenberg theory. For the Stokes system and general elliptic problems this symmetry is not granted and the full theory is needed.

The paper is organized as follows. We begin with a brief summary of elementary properties of generalized Lebesgue and Sobolev spaces which we will need in the sequel, and we introduce the concept of homogeneous Sobolev spaces in the variable exponent context, cf.~\cite{DieHHR10}. Then we state the generalizations of the classical \Calderon-Zygmund and Agmon-Douglis-Nirenberg theorems for symmetric kernels to generalized Lebesgue spaces. These generalizations have been treated for the first time in \Ruzicka, Diening \cite{DieR03half}, \cite{DieR03half2} in a somewhat weaker form. Unfortunately, the requirements on the kernel in \cite{DieR03half}, \cite{DieR03half2} seem too restrictive for an application to the Stokes problem. With the help of the results of Cruz-Uribe et al~\cite{CruFMP06} on singular integrals with rough kernels the requirements can be relaxed sufficiently, cf. \cite{DieHHR10}. In the subsequent section we prove the existence and uniqueness of a strong solution in $W^{2,p(\cdot)}\times W^{1,p(\cdot)}$ of the Stokes problem in bounded domains with $C^{1,1}$-boundary, provided that the right-hand sides are in $L^{p(\cdot)}\times W^{1,p(\cdot)}$ and the boundary values are in $\trace(W^{2,p(\cdot)})$. Furthermore, we show an analoguous result for weak solutions in $W^{1,p(\cdot)}\times L^{p(\cdot)}$ for the right-hand sides in $W^{-1,p(\cdot)}\times L^{p(\cdot)}$ and boundary values in $\trace(W^{1,p(\cdot)})$. The main idea of the proof is a localization technique to reduce the interior and the boundary regularity to regularity results on the whole-space and the half-space, respectively. In the final section we state the analoguous results for the Poisson problem, omitting the proofs.

\section{Variable exponent spaces}
\label{sec:preliminaries}

Let us introduce the variable exponent spaces $L^{p(\cdot)}(\Omega)$ and
$W^{k,p(\cdot)}(\Omega)$. Most of the following fundamental properties of
these spaces can be found in \cite{KovR91}, \cite{FanZ01}. We also refer to the extensive book \cite{DieHHR10} on variable exponent spaces.
Let $\Omega\subseteq\Bbb R^n$ be a domain. A measurable function $p:\Omega\to[1,\infty)$ is
called \emph{exponent}. If $p^+:=\sup p<\infty$, then $p$ is called
\emph{bounded exponent}.  For a bounded exponent $p$ we define
$L^{p(\cdot)}(\Omega)$ to consist of measurable functions
$f:\Omega\to\Bbb R$ such that the \emph{modular}
$$
\rho_p(f):=\int_\Omega|f(x)|^{p(x)}\,dx                                                                                                         
$$
is finite. The expression
$$
\|f\|_{p(\cdot)}:=\inf\{\lambda>0 : \rho_p(f/\lambda)\le 1\}                    
$$
defines a norm on $L^{p(\cdot)}(\Omega)$. This makes $L^{p(\cdot)}(\Omega)$ a Banach space. Moreover, one can show that $C^\infty_0(\Omega)$ is dense in $L^{p(\cdot)}(\Omega)$ and that $L^{p(\cdot)}(\Omega)$ is separable. Further, let $W^{k,p(\cdot)}(\Omega)$ denote the space of measurable functions $f:\Omega\to\Bbb R$ such that $f$ and the distributional derivatives $f,\nabla f,\ldots,\nabla^k f$ are in $L^{p(\cdot)}$. The norm $\|f\|_{k,p(\cdot)}:=\sum_{i=0}^k\|\nabla^k f\|_{p(\cdot)}$ makes $W^{k,p(\cdot)}(\Omega)$ a Banach space. By $W_0^{k,p(\cdot)}(\Omega)$ we denote the closure of $C^\infty_0(\Omega)$ in $W^{k,p(\cdot)}(\Omega)$. The space $W^{-k,p(\cdot)}(\Omega)$ is defined as the dual of the space $W^{k,p'(\cdot)}_0(\Omega)$. As usual we set $1/p + 1/{p'} = 1$. If $p^-:=\inf p>1$, then $W^{k,p(\cdot)}(\Omega)$ is reflexive. For bounded domains $\Omega\subset\setRR$ with Lipschitz-continuous boundary, we define the \emph{trace space} $\trace({W^{k,\pp}(\Omega)})$ by
$$\trace({W^{k,\pp}(\Omega)}):=\{f\in L^1(\pa \Omega)\ |\ \exists u\in W^{k,\pp}(\Omega): u|_{\pa \Omega} = f \}.$$
Then 
$$\norm{f}_{\trace({W^{k,\pp}(\Omega)})}:=\inf_{u\in W^{k,\pp}(\Omega), \atop u|_{\pa \Omega}=f}\norm{u}_{W^{k,\pp}(\Omega)}$$
defines a norm on $\trace({W^{k,\pp}(\Omega)})$ which makes the trace space a Banach space.

We have to impose some (weak) conditions on the exponent to recover important results from the classical Lebesgue and Sobolev spaces. The crucial condition is the so-called \emph{log-H{\"o}lder} continuity of the exponent $p$, i.e.,
$$
|p(x)-p(y)|\le\frac C{\ln(e+|x-y|^{-1})}                                                                                             \
$$
for all $x,y\in \Omega$. If $\Omega$ is unbounded, then this local continuity is supplemented by the condition that there exists the limit $p(\infty):=\lim_{x\to\infty}p(x)$ and
$$
|p(x)-p(\infty)|\le\frac C{\ln(e+|x|)}.                                                                                                         
$$
Let us denote by $\PPln(\Omega)$ the set of exponents satisfying the above conditions. If the exponent is in $\PPln(\Omega)$ then
$C^\infty(\overline{\Omega})$ is dense in $W^{1,p(\cdot)}(\Omega)$ for domains $\Omega$ with Lipschitz-continuous boundary. Let us now state some further results which will be needed later. 
The omitted proofs can be found for example in \cite{KovR91}, \cite{Die07habil}, \cite{DieHHR10}.

\begin{theorem}
Let $p$ be a bounded exponent in $\Omega$. Then the mapping $I: L^{p'(\cdot)}(\Omega) \rightarrow (\lpo)^*$ with $\langle If, g \rangle := (f,g)$
is an isomorphism, and for all $f\in L^{p'(\cdot)}(\Omega)$ we have
$$\frac{1}{2}\norm{f}_{L^{p'(\cdot)}(\Omega)}\le\norm{If}_{(\lpo)^*}\le 2\norm{f}_{L^{p'(\cdot)}(\Omega)}.$$
\label{ppdual}
\end{theorem}

\begin{theorem}
Let $p,q$ and $s$ be bounded exponents in $\Omega$ with $\frac{1}{s}=\frac{1}{p}+\frac{1}{q}$. For all $f\in\lpo$ and $g\in L^{q(\cdot)}(\Omega)$ we have
$fg\in L^{s(\cdot)}(\Omega)$ and 
$$\norm{fg}_{L^{s(\cdot)}(\Omega)} \le 2 \norm{f}_{L^{p(\cdot)}(\Omega)}\norm{g}_{L^{q(\cdot)}(\Omega)}.$$
\label{hoelder}\end{theorem}

\begin{theorem}
Let $\Omega\subset\setRR$ be a bounded domain and $p$ a bounded exponent in $\Omega$. Then:
\ben
\item For every exponent $q$ with $q\le p$ a.e., the embedding $$\lpo \embedding L^{q(\cdot)}(\Omega)$$ is continuous.
\item If moreover $\partial \Omega$ is Lipschitz-continuous and $p\in\PPln(\Omega)$ with $1< p^- \le p^+ <\infty$, then the embedding $$W_0^{1,\pp}(\Omega)\compactembedding\lpo$$ is compact.
\een
\label{einbettungen}
\end{theorem}
The following extension result can be found in~\cite{CruFMP06}, \cite{DieF08}, \cite{DieHHR10}.

\begin{theorem}
Let $\Omega\subset\setRR$ be a bounded domain with Lipschitz-continuous boundary and $p\in\PPln(\Omega)$ with $1< p^- \le p^+ <\infty$. Then there exists an exponent $\tilde{p}\in\PPln(\setRR)$ with $\tilde{p}|_{\Omega}=p$ and $\tilde{p}^+=p^+,\ \tilde{p}^-=p^-$ as well as an extension operator
$$\mathcal{E}: \wepo \rightarrow W^{1,\tilde{p}(\cdot)}(\setRR),\ (\mathcal{E}f)|_\Omega = f.$$
\label{expoext}
\end{theorem}

When working with partial differential equations in unbounded domains, as we will have to later on, it is often not natural to assume that the solution and its derivatives belong to the same Lebesgue space. For this reason we now introduce the \emph{homogeneous Sobolev spaces}. Let us present in the following the basic facts on those spaces. For details and proofs we refer to \cite{Len08}, \cite{DieHHR10}.

For a bounded exponent $p$ in $\Omega$, and $k\in \setN$ we define
\begin{align*}
  \widetilde D^{k,p(\cdot)}(\Omega):= \set{u \in L^{1}_\loc
    (\Omega) : \nabla ^k u \in L^{p(\cdot)}(\Omega)}\,.
\end{align*}
The linear space $\widetilde D^{k,p(\cdot)}(\Omega)$ is equipped with
the seminorm
\begin{align*}
  \|u\|_{\widetilde D^{k,p(\cdot)}(\Omega)} := \|\nabla ^k u
  \|_{L^{p(\cdot)}(\Omega)} \,.
\end{align*}
Note, that $\|u\|_{\widetilde D^{k,p(\cdot)}(\Omega)} =0$ implies that
$u $ is a polynomial of degree $k-1$. Let us denote the polynomials af
degree $m \in \setN_0$ by $\sf P_m$. It is evident
that the seminorm $\| \cdot\|_{\widetilde D^{k,p(\cdot)}(\Omega)} $
becomes a norm on the equivalence classes $[u]$ defined for $u \in
\widetilde D^{k,p(\cdot)}(\Omega) $ by
\begin{align*}
  [u]_{k-1}:= \set{ w \in \widetilde D^{k,p(\cdot)}(\Omega): w = u +
    p_{k-1} \textrm{ for some } p_{k-1} \in {\sf P_{k-1}} }\,.
\end{align*}

\begin{definition}\label{def:homo}
  Let $p$ be a
  bounded exponent in $\Omega$, and $k\in \setN$.  The {\em homogeneous
    Sobolev space} $D^{k,p(\cdot)}(\Omega)$ consists of all equivalence
  classes $[u]$ where $u \in \widetilde D^{k,p(\cdot)}(\Omega)$. We
  identify $u $ with its equivalence class $[u]$ and thus write $u$
  instead of $[u]$. The space $D^{k,p(\cdot)}(\Omega)$ is equipped
  with the norm
  \begin{align*}
    \|u\|_{D^{k,p(\cdot)}(\Omega)}:= \|\nabla ^k u
    \|_{L^{p(\cdot)}(\Omega)} \,.
  \end{align*}
  Finally, we define the space $D^{k,p(\cdot)}_0(\Omega)$ as the completion of $C^\infty _0 (\Omega)$ with respect to the norm 
  ${\|\cdot\|_{D^{k,p(\cdot)}(\Omega)}}$.
\end{definition}

\begin{remark}\label{rem:subspace}
  The natural embedding $i\colon C^\infty _0 (\Omega)
  \to D^{k,p(\cdot)}(\Omega)\colon u\mapsto [u] $ implies that $C^\infty
  _0 (\Omega)$ is isomorphic to a linear subspace of
  $D^{k,p(\cdot)}(\Omega)$. Consequently we can view
  $D^{k,p(\cdot)}_0(\Omega)$ as a subspace of
  $D^{k,p(\cdot)}(\Omega)$.
\end{remark}

\begin{theorem}\label{HSisBanach}\label{HSisReflexive}\label{HSisuconvex}
  The spaces $D^{k,p(\cdot)}(\Omega)$ and $D_0^{k,p(\cdot)}(\Omega)$
  are separable Banach spaces which are reflexive if $1<p^-\leqslant p^+<\infty$.
\end{theorem}

For an integrable function $u$ we define the \emph{mean value} of $u$ by $u_\Omega:=\dashint_\Omega u\, dx$. The spaces $\widetilde D^{k,p(\cdot)}(\Omega)$ and $W^{k,p(\cdot)}(\Omega)$
essentially do not differ for bounded domains. More precisely we have:
\begin{theorem}\label{pro:D=W}
  Let $\Omega $ be a bounded domain with Lipschitz continuous boundary, and let $p
  \in \PPln(\Omega)$ satisfy $1<p^-\leqslant p^+<\infty$. Then we have the
  algebraic identity 
  \begin{align*}
    \widetilde D^{k,p(\cdot)}(\Omega) = W^{k,p(\cdot)}(\Omega)\,.
  \end{align*}
  Moreover for $u \in \widetilde D^{1,p(\cdot)}(\Omega) $ we have the \Poincare\
  inequality 
  \begin{align}    \label{eq:poin-D}
    \norm{u - u_\Omega}_{L^{p(\cdot)}(\Omega)} & \le  c\,
    \diameter(\Omega)\, \norm{\nabla
      u}_{L^{p(\cdot)}(\Omega)} \,
  \end{align}
  with a constant $c$  depending on $n$, the Lipschitz constant, and the
  $\log$-H{\"o}lder constants of~${p}$.
\end{theorem}
%


\begin{remark}\label{rem:wkp_loc}
  As a consequence of the above theorem we get the algebraic identity $$\widetilde D^{k,p(\cdot)}(\Omega)= W_\loc^{k,p(\cdot)}(\Omega)$$ for arbitray domains provided that $p \in \PPln$ satisfies $1<p^-\leqslant p^+<\infty$.
\end{remark}

As for classical Sobolev spaces we have that $D^{k,p(\cdot)}( \setR^n
)$ and $D^{k,p(\cdot)}_0( \setR^n )$ coincide.

\begin{lemma}\label{pro:D0=D}
  Let $p \in \PPln(\setR^n)$ satisfy $1<p^-\leqslant p^+<\infty$ and
  let $k \in \setN_0$. Then $C^\infty_0 (\setR^n)$ is dense in
  $D^{k,p(\cdot)}( \setR^n )$. Consequently we have 
  $D^{k,p(\cdot)}( \setR^n )=D^{k,p(\cdot)}_0( \setR^n )$. 
\end{lemma}

In applications it happens that for a
function $u \in L^1_\loc(\Omega)$ one can show that $\nabla u \in
L^{p(\cdot)}(\Omega)$ and $\nabla^2 u \in L^{p(\cdot)}(\Omega)$. This information is neither
covered by the space $D^{1,p(\cdot )}(\Omega)$ nor by the space
$D^{2,p(\cdot )}(\Omega)$. Thus we introduce a new space containing
the full information. 

\begin{definition}\label{def:intersect-hom}
  Let $p$ be a bounded exponent in $\Omega$. The space $D^{(1,2),p(\cdot)}(\Omega)$
  consists of all equivalence classes $[u]_0$ with $u \in \widetilde
  D^{1,p(\cdot)}(\Omega)\cap \widetilde D^{2,p(\cdot)}(\Omega)$. We
  identify $u$ with its equivalence class $[u]_0$ and thus write $u$
  instead of $[u]_0$.  We equip the space $D^{(1,2),p(\cdot)}(\Omega)$
  with the norm
  \begin{align*}
    \|u\|_{D^{(1,2),p(\cdot)}(\Omega)}:=
    \|\nabla  u \|_{L^{p(\cdot)}(\Omega)} +\|\nabla ^2 u
    \|_{L^{p(\cdot)}(\Omega)}\,.
  \end{align*}
  Finally, we define the space $D^{(1,2),p(\cdot)}_0(\Omega)$ as the completion of $C^\infty _0 (\Omega)$ with respect to the norm
  ${\|\cdot\|_{D^{(1,2),p(\cdot)}(\Omega)}}$.
\end{definition}

Note that the space $D^{(1,2),p(\cdot)}(\Omega)$ is a subspace of
$D^{1,p(\cdot )}(\Omega)$ but not of $D^{2,p(\cdot
  )}(\Omega)$, because it consists of equivalence classes modulo
constants.  As in Remark~\ref{rem:subspace} one sees that
$D^{(1,2),p(\cdot)}_0(\Omega)$ can be viewed as a subspace of
$D^{(1,2),p(\cdot)}(\Omega)$. 

\begin{theorem}\label{IHSisBanach}\label{IHSisReflexive}\label{IHSisuconvex}
  The spaces $D^{(1,2),p(\cdot)}(\Omega)$, and
  $D^{(1,2),p(\cdot)}_0(\Omega)$ are separable Banach spaces which are reflexive if  $1<p^-\leqslant p^+<\infty$.
\end{theorem}

\begin{lemma}\label{pro:dense-D1-D2}
 Let $p \in \PPln(\setR^n)$ satisfy $1<p^-\leqslant p^+<\infty$. Then $C^\infty_0 (\setR^n)$ is dense in
 $D^{(1,2),p(\cdot)}(\setRR)$. Consequently we have $D^{(1,2),p(\cdot)}(\setRR)=D_0^{(1,2),p(\cdot)}(\setRR)$.
\end{lemma}

We will also need the \emph{dual spaces} of homogeneous Sobolev spaces.
\begin{definition}\label{def:dualD}
  Let $p$ be a bounded exponent in $\Omega$, and let $k \in \setN_0$. The
  space $D^{-k,p(\cdot)}(\Omega)$ is defined as the dual of the space $D^{k,p'(\cdot)}_0(\Omega)$, 
  i.e.~$D^{-k,p(\cdot)}(\Omega):=(D^{k,p'(\cdot)}_0(\Omega))'$. 
\end{definition}

We let $L^{p(\cdot)}_0(\Omega)\subset L^{p(\cdot)}(\Omega)$ be the subspace of functions having vanishing mean value. Analogously we define the subspace $C_{0,0}^\infty(\Omega)\subset\tfo$.
\begin{lemma}\label{cor:C00-D-1}
  Let $p \in \PPln(\Omega)$ satisfy $1<p^-\leqslant p^+<\infty$, and $A\subset\Omega$ be a bounded domain with Lipschitz continuous boundary. Then $L^{p(\cdot)}_0(A)\embedding D^{-1,p(\cdot)}(\Omega)$ via $\skp{f}{u}=\int_\Omega f u \, dx$ for $f\in L^{p(\cdot)}_0(A)$ and $u\in\tfo$.
\end{lemma}

\begin{proof}
Using Theorem \ref{pro:D=W} we get
\begin{equation}\bal\label{homest}
\skp{f}{u} = \int_{\Omega}f(x)(u(x)-u_A)\ dx &\le 2 \norm{f}_{L^{p(\cdot)}(\Omega)}\norm{u-u_{A}}_{L^{p'(\cdot)}(A)} \\
&\le c \norm{f}_{\lpr}\norm{\nabla u}_{L^{p'(\cdot)}(\Omega)}.
\eal\end{equation}
\end{proof}

If the domain $\Omega$ has a sufficiently large and nice boundary it
is not necessary to require as in the previous lemma that the function
$f$ has a vanishing mean value. For simplicity we formulate the result
only for the case of the half-space $\setR^n_>:=\{x\in\setRR| x_n>0\}$. We define $\setR^n_<$ accordingly and set $\Sigma:=\partial\setRR_>$.

\begin{lemma}\label{lem:C0-D-1}
  Let $p \in \PPln(\setR^n_>)$ satisfy $1<p^-\leqslant p^+<\infty$ and $A\subset\setR^n_>$ be a bounded domain. Then $L^{p(\cdot)}(A)\embedding D^{-1,p(\cdot)}(\setR^n_>)$. 
\end{lemma}
\begin{proof}
 We choose a ball $B(x_0)\subset\setRR$ with $A\subset B(x_0)$ and $x_0\in\Sigma$. Next we choose a ball $B'\subset\setRR_<\cap B(x_0)$ with $|B'|\approx|A|$ and note that $u_{B'}=0$ for $u\in C_0^\infty(\setRR_>)$. Now we use a version of the \Poincare\ inequality (cf. \cite[Lemma 7.2.3 and  Theorem 7.2.4 (b)]{DieHHR10}) to obtain for all $u\in C_0^\infty(\setRR_>)$
\begin{equation}\bal\label{homest2}
\skp{f}{u} =\int_{B(x_0)}f(x)u(x)\ dx &\le 2 \norm{f}_{L^{p(\cdot)}(B(x_0))}\norm{u-u_{B'}}_{L^{p'(\cdot)}(B(x_0))} \\
&\le c \norm{f}_{L^{p(\cdot)}(A)}\norm{\nabla u}_{L^{p'(\cdot)}(\setRR_>)}.
\eal\end{equation}
\end{proof}

\begin{lemma}\label{pro:C00-denseD-1}
  Let $p \in \PPln(\Omega)$ satisfy $1<p^-\leqslant p^+<\infty$. Then
  $C_{0,0}^\infty (\Omega)$ is dense in $L^{p(\cdot)}(\Omega)$ and in $D^{-1,p(\cdot)}(\Omega)$.
\end{lemma}

We will also have to deal with trace
spaces of homogenous Sobolev spaces, at least in the
case of the half-space $\setR^n_>$. Traces are well defined for functions from $W^{1,1}_\loc$ and thus
the notion of a trace of a function from $\widetilde
  D^{1,p(\cdot)}(\setRR_>)$
is well defined. Consequently, the \emph{trace space} $\trace (
  D^{1,p(\cdot)}(\setR^n_>))$ consists of the (equivalence classes of) traces
of all (equivalence classes) $f \in D^{1,p(\cdot)}(\setR^n_>)$.  The norm
\begin{align*}
  \norm{f}_{\trace( D^{1,p(\cdot)}(\setR^n_>))} := \inf \bigset{
    \norm{g}_{D^{1,p(\cdot)}(\setR^n_>)} \colon g \in
    D^{1,p(\cdot)}(\setR^n_>) \text{ and } \trace g = f}
\end{align*}
makes $\trace (D^{1,p(\cdot)}(\setR^n_>))$ a Banach space. 

Another important issue is the extension of the theory of singular integrals to variable exponent spaces. We refer to \cite{CruFMP06}, \cite{DieHHR10} for proof of the following two theorems. Firstly, we will need the generalization of the classical \Calderon-Zygmund theorem. We will restrict ourselves to symmetric kernels, i.e., to kernels, depending only on the difference of their arguments. A (symmetric) \emph{kernel} $K$ in $\Omega$, $\Omega\subset\setRR$, is a locally integrable, real-valued funktion in $\Omega\setminus\{0\}$. 

\begin{theorem} Let $K$ be a kernel in $\setRR$ of the form
\begin{equation}
K(x)=\frac{P(x/|x|)}{|x|^n},
\label{kernform}
\end{equation}
where $P\in L^r(\pa B_1(0))$ for some $r\in(1,\infty]$, and satisfying
\begin{equation}
\int_{\pa B_1(0)}P\ d\omega = 0.
\label{intnull}
\end{equation}
Moreover let $p\in\PPln(\setR^n)$ be bounded with $p^->r'$.
Then the operator $T$, defined by
$$(Tf)(x):=\lim_{\epsilon\searrow 0}\int_{(B_\epsilon(x))^c} K(x-y)f(y)\ dy$$
is bounded on $\lpr$.
\label{hsabschgr}
\end{theorem}

Furthermore we will need the analogue of the famous Agmon-Douglis-Nirenberg result for spaces with variable exponents. To this end, let us fix some notation. For $x=(x_1,\ldots,x_n)\in\setRR$ we set $x':=(x_1,\ldots,x_{n-1})\in\setR^{n-1}$ and $\tilde x:=(x_1,\ldots,x_{n-1},-x_n)$. Moreover we set $S:=\pa B_1$, $S^{n-2}:=\pa B_1^{n-1}\subset\setR^{n-1}$, $S_> :=S\cap\setRR_>$, and $S_\ge :=S\cap\setRR_\ge$ where $\setR^n_\ge:=\{x\in\setRR| x_n\ge 0\}$.

\begin{theorem}\label{hsabschhr}
Let  $K$ be a kernel  on $\setR^n_\ge $ of the form
\begin{equation*}
  K(x)=\frac{P(x/|x|)}{|x|^n},
\end{equation*}
where $P\colon S_\ge \to \setR$ is continuous and satisfies 
\begin{equation*}
  \int_{S^{n-2}}P(x',0)\ d\omega' = 0.
\end{equation*}
Assume that $K$ possesses continuous derivatives $\pa_i K$,
$i=1,\ldots,n$, and $\pa_n^2 K$ in $\setRR_>$ which are bounded on
the hemisphere $S_>$. Let $p \in \PPln(\setRR_>)$ satisfy $1< p^- \le p^+
<\infty$. For $f \in C^\infty_0(\setRR_\ge)$ we define $Hf\,:\,
\setRR_> \to \setR$ by
\begin{equation}
  (Hf)(x):=\int_\Sigma K(x'-y',x_n)f(y',0)\ dy'.
  \label{randop}
\end{equation}
Then $H$ satisfies
\begin{align}
  \norm{\nabla Hf}_{L^{\pp}(\setRR_>)}\le
  c\, \norm{\nabla f}_{L^{\pp}(\setRR_>)}, \label{est-randop}
\end{align}
for all $f \in C^\infty_0(\setRR_\ge)$ with a constant $c=c(p,n,P)$.
In particular, $H$ extends to a bounded linear operator
$H \colon D^{1,p(\cdot)}(\setRR_>) \to D^{1,p(\cdot)}(\setRR_>)$.
\end{theorem}

Throughout the paper we will make use of Einsteins summation convention, i.e. whenever there is an index appearing twice in a monomial this implies that we are summing over all of its possible values.

\section{Stokes system}
\label{sec:Main}

In this section we assume that $\Omega{}$ is a bounded domain
in $\setR^n$, $n \ge 2$, with $C^{1,1}$-boundary. We want to show that the
Stokes system 
\begin{equation}\label{eq:stokes}
\begin{aligned}
\Delta \bv - \nabla \pi &=\ff & &\mbox{in } \Omega,\\
\divergenz \bv &=g & &\mbox{in } \Omega,\\
\bv&=\bv_0 & &\mbox{on } \partial\Omega,
\end{aligned}
\end{equation}
possesses a unique strong and weak solution, respectively, provided that the data have appropriate
regularity. More precisely we prove:

\begin{theorem}\label{thm:stokes-strong}
  Let $p \in \PPln(\Omega)$ satisfy $1< p^- \le p^+ <\infty$. For arbitrary data $\ff \in (L^{p(\cdot)}(\Omega))^n$, $g \in
  W^{1,\pp}(\Omega)$ and $\bv_0 \in \trace (W^{2,p(\cdot
    )}(\Omega))^n$ satisfying the compatibility condition $
  \int\nolimits_\Omega g\ dx =\int\nolimits_{\pa\Omega} \bv_0 \cdot
  \bfnu\ d\omega $ there exists a unique strong solution $(\bv,\pi)
  \in (\wzpo)^n \times \wepo$ of the Stokes system \eqref{eq:stokes} with $\int\nolimits_\Omega \pi \ dx=0$ and which satisfies the estimate
  \beq\bal
    \normzpo{\bv} + &\normepo{\pi} \\
    &\le c \left(\normpo{\ff} +
      \norm{{\bv}_0}_{\trace (W^{2,p(\cdot
          )}(\Omega))} +\normepo{g} \right), 
  \eal\eeq
  where the constant $c$ depends only on the domain $\Omega$ and the
  exponent $p$.
\end{theorem}

\begin{theorem}\label{thm:stokes-weak}
  Let $p \in \PPln(\Omega)$ satisfy $1< p^- \le p^+ <\infty$. For arbitrary data $\ff \in (W^{-1,\pp}(\Omega))^n$, $g\in \lpo$ and $\bv_0\in \trace((\wepo)^n)$ satisfying the above compatibility condition there exists a unique weak solution $(\bv,\pi) \in (\wepo)^n\times \lpo$ of the Stokes system \eqref{eq:stokes} with $\int\nolimits_\Omega \pi \ dx=0$ and which satisfies the estimate
  \beq\bal
 \normepo{\bv} + &\normpo{\pi}\\ 
&\le c\left( \normnegpo{\ff} + \normpo{g} + \norm{\bv_0}_{\trace(\wepo)}\right),
  \eal\eeq
where the constant $c$ depends only on the domain $\Omega$ and the
  exponent $p$.
\end{theorem}

We call $(\bv,\pi)$ a \emph{strong solution} of \eqref{eq:stokes} provided that it satisfies the differential equations in \eqref{eq:stokes} in the sense of weak derivatives. Furthermore we call $(\bv,\pi)$ a \emph{weak solution} of \eqref{eq:stokes} provided that
\begin{equation*}
\begin{aligned}
\int_\Omega \nabla\bv : \nabla\bphi\ dx - \int_\Omega \pi \cdot\divergenz \bphi\ dx &= -\langle \ff,\bphi\rangle & &\forall\bphi\in(W_0^{1,\pps}(\Omega))^n,\\
\divergenz \bv &=g & &\mbox{in } \Omega,\\
\bv&=\bv_0 & &\mbox{on } \partial\Omega.
\end{aligned}
\end{equation*}
In fact it is sufficient to consider homogeneous boundary conditions, i.e.~$\bv _0 =\mathbf 0$. In
the general case, when considering strong solutions, we take a realization $\tilde\bv_0\in (\wzpo)^n$ of $\bv_0\in \trace((\wzpo)^n)$ and construct a strong solution $(\bv,\pi)$ for the data $\ff-\Delta \tilde\bv_0$, $g-\divergenz \tilde\bv_0$ and vanishing boundary values. Note that $\int_\Omega (g-\divergenz \tilde{\bv}_0) dx= 0$. Then $(\bv:=\bu + \tilde{\bv}_0,\pi)$ is the unique strong solution, satisfying the assertions of the theorem. When dealing with weak solutions we may proceed in a very similar way. 

Although the proof of Theorems \ref{thm:stokes-strong} and \ref{thm:stokes-weak} is based on the classical, i.e. constant exponent, theory of the Stokes system (cf. \cite{Gal94}), we will nevertheless have to follow the strategy of the proof of the classical case all over again, i.e., we will use a localisation technique to reduce the problem in general bounded domains to the problem in the whole-space and in the half-space. The treatment of these situations is based on the \Calderon-Zygmund theory of singular integral operators and on the Agmon-Douglis-Nirenberg theory of operators in the half-space.

It should be emphasized that deriving analogous results for the Poisson equation is much simpler than our task, although the proof essentially follows the same idea. The crucial difference is that the fundamental solution of the Poisson equation in half-space is given as the whole-space fundamental solution plus its own odd reflection. This simplifies the half-space case considerably. Indeed, the Agmon-Douglis-Nirenberg theory is not needed in the treatment of the Poisson equation. An analogous ansatz in the case of the Stokes system yields non-divergence free functions und thus won't work. This is the reason for the need of the Agmon-Douglis-Nirenberg theory and homogenous Sobolev spaces in this paper.

Since the structure of the fundamental solutions of the Stokes system is different for $n=2$ and $n\ge 3$, we restrict ourselves to the latter case. The methods presented here can be easily adapted to treat also the case $n=2 $. Furthermore, using Theorem \ref{expoext}, a well behaved exponent given on a bounded domain will always be extended to the whole of $\setRR$ without mentioning.

Solutions of the system
\begin{equation}\label{eq:stokes1}
\begin{aligned}
\Delta \bv - \nabla \pi &=\ff & &\mbox{in } \setR^n,\\
\divergenz \bv &=g & &\mbox{in } \setR^n,
\end{aligned}
\end{equation}
are obtained by a convolution of the {\em fundamental solutions} $\big(-\frac{1}{2\abs{\partial B_1}} V^{rl}\big)_{r,l=1,\ldots,n}$ and $\big(-\frac 1{\abs{\partial B_1}} Q^l\big)_{l=1,\ldots,n}$ of
the Stokes system in the whole-space, given by 
$$ V^{rl}(x) := \left (\frac{1}{n-2}\frac{\delta_{rl}}{|x|^{n-2}} +
  \frac{x_rx_l}{|x|^n}\right )\qquad \mbox{ and } \qquad Q^l(x):=  \frac{x_l}{|x|^n},$$
with data $\ff $ and $g$. From the classical
theory it is well known that the kernels $\pa_{i}\pa_j V^{rl}$ and
$\pa_iQ^l$ satisfy the assumptions of Theorem
\ref{hsabschgr}. Consequently we get:

\begin{lemma}\label{stabl}
  Let $p \in \PPln(\setR^n)$ satisfy $1< p^- \le p^+ <\infty$, and let
  $\ff\in(C^\infty_{0,0}(\setRR))^n$.  Then the convolutions
  $\bv(x):=\int\nolimits_{\setR^n}\bV(x-y)\ff(y)dy $ and
  $\pi(x):=\int\nolimits_{\setR^n}\bQ(x-y)\cdot\ff(y)dy $ are infinitely differentiable. Moreover, their first and second order
  derivatives have the representations $(i,j,r=1,\ldots,n)$
  \begin{equation*}\bal
    \partial_i v_r(x)&= \int_{\setR^n} \partial_{x_i}V^{rl}(x-y) f_l(y)\ dy,\\
    \partial_{i}\pa_{j}v_r(x)&= \lim_{\epsilon \searrow 0}
    \int_{(B(x,\epsilon))^c} \partial_{x_i}\pa_{x_j}V^{rl}(x-y)f_l(y)\ dy \\
    &\hspace{1cm}+ \frac{2|B_1|}{n+2}\big(-(n+1)\delta_{ij}f_{r}(x) +
      \delta_{ir}f_{j}(x) + \delta_{rj}f_{i}(x) \big),\hspace{-10mm}\\[1mm] 
    \pa_i\pi(x)&= \lim_{\epsilon \searrow 0}
    \int_{(B(x,\epsilon))^c} \partial_{x_i}Q^l(x-y)f_l(y)\ dy +
    |B_1|f_i(x), 
\label{stdarst}
\eal\end{equation*}
and satisfy the estimates 
\begin{equation}\bal
  \norm{\nabla \bv}_{L^{\pp}(\setRR)} + \norm{\pi}_{\lpr}&\le c
  \,\|\ff\|_{D^{-1,\pp}(\setR^n)}\,,
  \\ 
  \normpr{\nabla^2 \bv} + \normpr{\nabla \pi}&\le c\, \normpr{\ff}\,,
\label{stabsch}
\eal\end{equation}
with a constant  $c=c(p,n)$.
\end{lemma}

\begin{proof}
As in the classical theory (cf. \cite{Gal94}) we deduce the representations of the derivatives using integration by parts. These representations together with Theorem \ref{hsabschgr} immediately yield the estimate \eqref{stabsch}$_2$.

Using Theorem \ref{ppdual}, the density of $C_0^\infty(B)$ in $L^{p'(\cdot)}(B)$, the representation of $\nabla \bv$ and the Theorem of Fubini we may estimate the norm of the first order derivatives of $\bv$ on every Ball $B\subset\setR^n$ by
\beq\bal
\norm{\pa_i v_r}_{L^{\pp}(B)}&\le 2\sup_{\phi \in C_0^\infty(B), \atop \norm{\phi}_{L^{p'(\cdot)}(B)}\le 1} \int_B\phi\ \partial_i v_r\ dx\\
& = 2\sup_{\phi \in C_0^\infty(B), \atop \norm{\phi}_{L^{p'(\cdot)}(B)}\le 1} \int_{\setRR} f_l(y)\underbrace{\int_{\setR^n} \pa_{x_i}V^{rl}(x-y) \phi(x)\ dx}_{=:\Phi_l(y)} dy.
\eal\eeq
From the properties of $V$ and $\phi$ and from Theorem \ref{hsabschgr} we deduce the estimate
$$\norm{\nabla\Phi_l}_{L^{\pps}(\setRR)} \le c\norm{\phi}_{L^{\pps}(B)},$$
and thus
$$\norm{\pa_i v_r}_{L^{\pp}(B)}\le c\sup_{\bPhi \in (D^{1,p'(\cdot)}(\setRR))^n, \atop \norm{\nabla \bPhi}_{L^{p'(\cdot)}(\setRR)}\le 1} \int_{\setRR} \ff\cdot\bPhi\ dy
= c \|\ff\|_{D^{-1,\pp}(\setR^n)}$$
with a constant $c$ independent of $B$. The norm of $\pi$ may be dealt with in the same way. Hence we get \eqref{stabsch}$_1$.
\end{proof}

Note that from the representations of the derivatives it follows that the convolutions $\bv(x):=-\frac{1}{2|\pa B_1|}\int_{\setR^n}\bV(x-y)\ff(y)dy$ and $\pi(x):=-\frac{1}{|\pa B_1|}\int_{\setR^n}\bQ(x-y)\cdot\ff(y)dy$ solve the Stokes system in the whole-space, i.e. $\Delta \bv - \nabla \pi = \ff, \divergenz \bv = 0$ in $\setRR$.

Lemma \ref{stabl} and the density of $C^\infty
_{0,0}(\setRR)$ in $\lpr$ and $D^{-1,p(\cdot)}(\setRR)$ (cf. Lemma \ref{pro:C00-denseD-1}) show that the convolution with the kernel $\bV$ extends to a bounded operator $U$ from $(D^{-1,p(\cdot)}(\setRR))^n$ into $(D^{1,p(\cdot)}(\setRR))^n$ and from $(\lpr)^n$ into
$(D^{2,p(\cdot)}(\setRR))^n$. Similarly we see that the convolution with the kernel $\bQ$ extends to a bounded operator $P$ from $(D^{-1,p(\cdot)}(\setRR))^n$ into $L^{p(\cdot)}(\setRR)$ and from $(\lpr)^n$ into $D^{1,p(\cdot)}(\setRR)$. This proves the assertions (1) and (2) of the following theorem.

\begin{theorem}\label{cor:stabl} 
  Let $p \in \PPln(\setR^n)$ satisfy $1< p^- \le p^+ <\infty$ and $U,
  P$ be the operators defined above.
  \begin{enumerate}
  \item If $\ff\in (\lpr)^n$ then $U \ff \in (D^{2,p(\cdot)}(\setRR))^n$
    and $P\ff \in D^{1,p(\cdot)}(\setRR)$ satisfy the estimate
    \begin{equation*}
      \norm{U\ff}_{D^{2,p(\cdot)}(\setRR)} +
      \norm{P \ff }_{D^{1,p(\cdot)}(\setRR)} \le c\, \norm{\ff}_{L^{p(\cdot)}(\setRR)}
    \end{equation*}
    with a constant $c=c(p,n)$.
  \item If $\ff\in (D^{-1,p(\cdot)}(\setRR))^n$ then $U \ff \in
    (D^{1,p(\cdot)}(\setRR))^n$ and $P\ff \in L^{p(\cdot)}(\setRR)$
    satisfy the estimate
    \begin{equation*}
      \norm{U\ff}_{D^{1,p(\cdot)}(\setRR)} +
      \norm{P \ff }_{L^{p(\cdot)}(\setRR)} \le c\, \norm{\ff}_{D^{-1,p(\cdot)}(\setRR)}
    \end{equation*}
    with a constant $c=c(p,n)$.
  \item If $\ff\in(\lpr)^n$ has bounded support and vanishing mean value, hence $\ff\in (D^{-1,\pp}(\setRR))^n$, then
    $U\ff \in (D^{(1,2),p(\cdot)}(\setRR))^n\subset(D^{1,p(\cdot)}(\setRR))^n$ and $P\ff \in
    W^{1,p(\cdot)}(\setRR)\subset\lpr$ satisfy both of the above estimates simultaneously.
  \end{enumerate}
\end{theorem}

\begin{proof}
It remains to show (3). We consider the operators $U$ and $P$ from assertion (2). Due to the estimate \eqref{homest} any $\ff$ as in (3) approximated by some sequence $(\ff_k)\subset (C^\infty _{0,0}(\setRR))^n$ in $(L^{p(\cdot)}(\setRR))^n$ is approximated by the same sequence in $(D^{-1,\pp}(\setRR))^n$. From (1) and (2) thus follows that $(P\ff_k)$ is a Cauchy sequence in $W^{1,p(\cdot)}(\setRR)$ while $(U\ff_k)$ is a Cauchy sequence in $(D^{(1,2),p(\cdot)})^n$.
\end{proof}

\begin{corollary}\label{stmaingr} 
  Let $p \in \PPln(\setR^n)$ satisfy $1< p^- \le p^+ <\infty$, and let
  $\ff\in(\lpr)^n$ and $g \in W^{1,\pp}(\setR^n)$ have bounded
  support. Moreover, let $\ff$ have vanishing mean value and let
  $(\bv,\pi)\in (W^{2,p^-}(\setR^n))^n\times W^{1,p^-}(\setR^n)$ be a
  solution of the Stokes system \eqref{eq:stokes1}. Then the first
  and second weak derivatives of $\bv$ as well as $\pi$ and its first
  weak derivatives belong to the space $\lpr$. They satisfy the
  estimates 
\begin{equation}\bal
  \norm{\nabla \bv}_{L^{\pp}(\setRR)} + \norm{\pi}_{\lpr}&\le c
  \left(\|\ff\|_{D^{-1,\pp}(\setR^n)} + \norm{g}_{\lpr} \right),\\ 
  \normpr{\nabla^2 \bv} + \normpr{\nabla \pi}&\le c \left(\normpr{\ff}
    + \normpr{\nabla g} \right), 
\label{stmaingrabsch}
\eal\end{equation}
with a constant   $c=c(p.n)$.
\end{corollary}

\begin{proof}
From Section \ref{sec:poisson problem} we know that the function $$\bh:=-L(\nabla g) \in ((D^{(1,2),p(\cdot)})^n$$ satisfies the estimates
\begin{equation}\bal
\normpr{\nabla \bh} &\le c \normpr{g},\\
\normpr{\nabla^2 \bh} &\le c \normpr{\nabla g},
\label{habsch}
\eal\end{equation}
and the identities $\Delta \bh = \nabla g$ und $\divergenz \bh=g$ in $\setR^n$. Note that $\nabla g$ has vanishing mean value and $\|\nabla g\|_{D^{-1,\pp}(\setR^n)}\le \norm{g}_{\lpr}$. 

Set $\bF:=\ff - \Delta \bh$. The remark following Lemma \ref{stabl} and a density argument imply that $$\tilde{\bu}:=-\frac{1}{2|\pa B_1|}U\bF$$ and $$\tilde{\pi}:=-\frac{1}{|\pa B_1|}P\bF$$ solve the Stokes system in the whole-space, i.e. $\Delta \tilde{\bu} -\nabla \tilde{\pi} =\bF, \divergenz \tilde{\bu} = 0$ in $\setR^n$.
Setting $\tilde{\bv}:=\tilde{\bu} + \bh$ we conclude that
\beq\bal
\Delta \tilde{\bv} -\nabla \tilde{\pi} = \Delta \tilde{\bu} + \Delta \bh -\nabla \tilde{\pi} = \nabla \tilde{\pi} + \bF +\Delta \bh - \nabla \tilde{\pi} &= \ff,\\
\divergenz \tilde{\bv} &= g,
\eal\eeq
in $\setRR$. From theorem \ref{cor:stabl} and \eqref{habsch} we easily deduce the estimates \eqref{stmaingrabsch} for $\bv$ and $\pi$ replaced by $\tilde{\bv}$ and $\tilde{\pi}$.

Since $\ff$ and $g$ have bounded support analogous estimates hold with $\pp$ replaced by $p^-$. Using these estimates, the fact that solutions
$(\bv, \pi) \in D^{1,q}(\setRR) \times L^{q}(\setRR)$, $1<q<\infty$,
of the Stokes system \eqref{eq:stokes1} are unique up to a constant
(cf. \cite[Theorem IV.2.2]{Gal94}), and the integrability of $\pi $
and $\widetilde \pi $ we obtain \eqref{stmaingrabsch}. 
\end{proof}

Now we are ready to prove interior estimates for solutions of the
Stokes system.

\begin{theorem}\label{sttrafogr}  
  Let $p \in \PPln(\Omega)$ satisfy $1< p^- \le p^+ <\infty$, and let
  $\ff\in(L^{\pp}(\Omega))^n$ and $g \in W^{1,\pp}(\Omega)$. Let 
  $\Omega_0,\Omega_1 \subset \Omega$ be open sets with  $\Omega_0
  \subset\subset\Omega_1 \subset\subset\Omega$. Moreover, let 
  $(\bv,\pi)\in (W^{2,\pp}(\Omega))^n\times W^{1,\pp}(\Omega)$ be a
  solution of the Stokes system \eqref{eq:stokes}$_{1,2}$. Then there
  exists a constant $c=c(p,\Omega_0,\Omega_1)$ such that $(\bv,
  \pi)$ satisfy the estimates 
\begin{equation*}\bal
  &\norm{\nabla \bv}_{L^{\pp}(\Omega_0)} + \norm{\pi}_{L^{\pp}(\Omega_0)}\\
  &\hspace{5mm}\le c \big(\|\ff\|_{W^{-1,\pp}(\Omega_1)} + \norm{g}_{L^{\pp}(\Omega_1)} +
  \norm{\bv}_{L^{\pp}(\Omega_1 \setminus \Omega_0)} +
  \norm{\pi}_{W^{-1,\pp}(\Omega_1 \setminus \Omega_0)} \big),\\ 
  &\norm{\nabla^2 \bv}_{L^{\pp}(\Omega_0)} + \norm{\nabla \pi}_{L^{\pp}(\Omega_0)}\\
  &\hspace{5mm}\le c \big(\norm{\ff}_{L^{\pp}(\Omega_1)} + \norm{g}_{W^{1,\pp}(\Omega_1)} +
  \norm{\bv}_{W^{1,\pp}(\Omega_1 \setminus \Omega_0)} +
  \norm{\pi}_{L^{\pp}(\Omega_1 \setminus \Omega_0)}\big). 
\eal\end{equation*}
\end{theorem}

\begin{proof}
  Let $\tau\in C^\infty(\setRR)$ with $\tau=1$ in $\Omega_0$ und
  $\support(\tau)\subset\subset\Omega_1$. For $\bar{\bv}:=\bv\tau$ and
  $\bar \pi := \pi \tau$ we have 
  \beq\bal 
  \Delta \bar{\bv} - \nabla{\bar{\pi}}& =2\nabla\bv\nabla\tau +
  \Delta\tau\bv -
  \pi\nabla\tau + \ff\tau=:\bfT,\\
  \divergenz\bar{\bv} &= \bv \cdot\nabla\tau + g\tau =:G,
  \eal\eeq
in $\setRR$. Integrating by parts we obtain $$\int_{\setRR} \bT\ dx =  \int_{\setRR} \left(-\Delta \bv + \nabla \pi + \ff\right)\phi\ dx= 0.$$ Hence we see that $\bar \bv \in
  (W^{2,\pp}(\setRR))^n$, $\bar{\pi} \in W^{1,\pp}(\setRR)$, $G\in
  W^{1,\pp}(\setRR)$ and $\bfT \in (L^{\pp}(\setRR))^n$ satisfy the
  assumptions of corollary \ref{stmaingr} which yields
\begin{equation*}\hspace*{-1mm}\bal
  \norm{\nabla \bv}_{L^{\pp}(\Omega_0)} + \norm{\pi}_{L^{\pp}(\Omega_0)}&\le c
  \left(\|\bfT\|_{D^{-1,\pp}(\setR^n)} + \norm{G}_{L^{\pp}(\Omega_1)} \right),\\ 
  \norm{\nabla^2 \bv}_{L^{\pp}(\Omega_0)} + \norm{\nabla
    \pi}_{L^{\pp}(\Omega_0)}&\le c \left(\norm{\bfT}_{L^{\pp}(\Omega_1)} 
    + \norm{\nabla G}_{L^{\pp}(\Omega_1)} \right).
\eal\end{equation*}
The last estimate immediately yields the second estimate of the theorem. To prove the first one we use
\beq\bal
\|\bfT\|_{D^{-1,\pp}(\setR^n)} &= \sup_{\bPsi \in D^{1,\pps}_0(\setRR), \atop \norm{\nabla\bPsi}_{p'(\cdot),\setR^n}\le 1} \int_{\setRR} \left( -\bv \cdot \bPsi \Delta \tau
-2 \bv \cdot\nabla\bPsi\nabla\tau - \pi\nabla\tau\cdot\bPsi + \ff\cdot\bPsi\tau\right)\ dx\\
&\le c \sup_{\bPsi \in D^{1,\pps}_0(\setRR), \atop \norm{\nabla\bPsi}_{p'(\cdot),\setR^n}\le 1} \Big(\norm{\bv}_{L^{\pp}(\Omega_1\setminus\Omega_0)}\norm{\bPsi\Delta\tau}_{L^{\pps}(\Omega_1\setminus\Omega_0)}\\
&\hspace{3cm}+ \norm{\bv}_{L^{\pp}(\Omega_1\setminus\Omega_0)}\norm{\nabla\bPsi\nabla\tau}_{L^{\pps}(\Omega_1\setminus\Omega_0)}\\
&\hspace{3cm}+ \norm{\pi}_{W^{-1,\pp}(\Omega_1\setminus\Omega_0)}\norm{\bPsi\cdot\nabla\tau}_{W^{1,\pps}(\Omega_1\setminus\Omega_0)}\\
&\hspace{3cm}+ \norm{\ff}_{W^{-1,\pp}(\Omega_1)}\norm{\bPsi\tau}_{W^{1,\pps}(\Omega_1)}\Big).
\eal\eeq
We integrated by parts to derive the equality. Afterwards we restricted the domains of integration appropriately and used H�lder's inequality. Since $\bT$ has vanishing mean value we may assume that $\bPsi$ has vanishing mean value in $\Omega_1$. Hence, using theorem \ref{pro:D=W} we may estimate the $\pps$-norm of the cut-off $\bPsi$ in $\Omega_1$ by $\norm{\nabla\bPsi}_{\pps,\setRR}$. This finishes the proof.
\end{proof}

Now we turn our attention to the \emph{Stokes system in the
  half-space} 
\begin{equation}\label{eq:stokes2}
\begin{aligned}
\Delta \bv - \nabla \pi &=\ff & &\mbox{in } \setR^n_>,\\
\divergenz \bv &=g & &\mbox{in } \setR^n_>,\\
\bv&= \mathbf 0  & &\mbox{on }\Sigma.
\end{aligned}
\end{equation}
In order to derive estimates for this problem we reflect the data in an even
manner and, by a convolution with the fundamental solutions of
the Stokes system, produce a whole-space solution $\widetilde\bv$ of \eqref{eq:stokes1}.
This solution does not satisfy the homogeneous boundary condition $\widetilde\bv
=\mathbf 0$ on $\Sigma$. To achieve this we add to
$\widetilde\bv$ a solution of the problem in the half-space
\begin{equation}\bal
\Delta \bw -\nabla \nu &= \mathbf 0 & &\mbox{ in } \setR^n_>,\\
\divergenz \bw &= 0 & &\mbox{ in } \setR^n_>,\\
\bw &= \bfh & &\mbox{ on } \Sigma,
\label{randsystem}
\eal\end{equation}
with the special choice $\bfh =-\widetilde\bv|_\Sigma$. In order to obtain appropriate
estimates of solutions of \eqref{randsystem} we need Theorem \ref{hsabschhr}. The solutions of
this problem are obtained as usual by a convolution of the normal derivatives of the fundamental solutions 
in the half-space, namely $\bfZ=(Z^{rl})_{r,l=1,\ldots, n}$ and $(\pa_l z)_{l=1,\ldots, n}$ with
\begin{align*}
  Z^{rl}(x) &:= \frac{2}{|B_1|}\frac{x_nx_rx_l}{|x|^{n+2}},\\
\intertext{and}
 &z(x):=-\frac{4}{|\pa B_1|} \frac{x_n}{|x|^n},
\end{align*}
with the boundary data $\bfh$ from \eqref{randsystem}. From the
classical theory it is well known that the kernels $Z^{rl}$ and $z$
satisfy the assumptions of Theorem \ref{hsabschhr}. Thus we obtain:

\begin{theorem}\label{randsatz}
  Let $p \in \PPln(\setRR)$ satisfy $1< p^- \le p^+ <\infty$, and let $\bfh \in
  ( D^{(1,2),p(\cdot)}(\setRR))^n$.  Then there exists a solution
$$
  (\bw,\nu)\in (D^{(1,2),p(\cdot)}(\setRR_>))^n\times
  W^{1,p(\cdot)}(\setRR_>)
$$ 
of the Stokes system in the half-space \eqref{randsystem} with
boundary data $\bfh|_\Sigma$ which satisfies the estimates
\begin{equation}\begin{aligned}
  \norm{\nabla \bw}_{L^{p(\cdot)}(\setRR_>)} + \norm{\nu}_{L^{p(\cdot)}(\setRR_>)} &\le
  c\,\norm{\nabla \bfh}_{L^{p(\cdot)}(\setRR_>)},\\ 
  \norm{\nabla^2 \bw}_{L^{p(\cdot)}(\setRR_>)} + \norm{\nabla \nu}_{L^{p(\cdot)}(\setRR_>)} &\le
  c\, \norm{\nabla^2 \bfh}_{L^{p(\cdot)}(\setRR_>)},
  \label{randabsch}
\end{aligned}\end{equation}
with a constant $c=c(p,n)$. 
\end{theorem}

\begin{proof}
  Since $\tfr$ is dense in $D^{(1,2),\pp}(\setRR )$ it suffices to show the existence of a linear solution operator defined on $\tfr$ and satisfying the estimates \eqref{randabsch}. Thus let $\bfh \in (\tfr)^n$. Then the functions $\bw:\setR^n_>\rightarrow\setR^n$ and
  $s:\setR^n_>\rightarrow \setR$ defined by
  \beq\bal
    \bw(x)&:=\int_\Sigma \bfZ(x'-y',x_n)\bfh(y',0)\ dy',\\
    \nu(x)&:=\sum_{i=1}^n \pa _i\underbrace{ \int_\Sigma
      z(x'-y',x_n)h_i(y',0)\ dy'}_{=:\nu_i(x)}, 
  \eal\eeq 
  are smooth solutions of \eqref{randsystem}. Indeed, it is easy to see that $\bw$ and $\nu$ are infinitely differentiable und that the derivatives may be written as
\begin{equation}\bal
\pa^\alpha\bw(x)=\int_\Sigma \pa^\alpha_x \bZ(x-(y',0))\ff(y')\ dy',\\
\pa^\alpha \nu(x)=\int_\Sigma\pa^\alpha_x\pa_{x_i} z(x-(y',0))f_i(y')\ dy'.
\label{wsabl}
\eal\end{equation}
Using these representations and the fact that $\Delta Z^{rl}-\pa^2_{rl}z= 0$ and $\pa_r Z^{rl}= 0$ in $\setRR\setminus\{0\}$ we conclude that $\bw$ and $\nu$ solve the Stokes system \eqref{randsystem}$_{1,2}$.
In order to show that the boundary values are met continuously we assume that $x\in\setR^n_>$ and $\epsilon>0$. It is not hard to see that
$$\int_{B^{n-1}_\epsilon(x')}Z^{rl}(x-(y',0))\ dy' = \delta_{rl} + o(1)$$
for $x_n \searrow 0$. Thereby we get
\beq\bal
w_r(x)-f_r(x') &= \int_{B^{n-1}_\epsilon(x')}Z^{rl}(x-(y',0))f_l(y')\ dy' - f_r(x')\\
&\hspace{0.5cm} + \int_{\Sigma\setminus B^{n-1}_\epsilon(x')}Z^{rl}(x-(y',0))f_l(y')\ dy'\\
&= \int_{B^{n-1}_\epsilon(x')}Z^{rl}(x-(y',0))(f_l(y')-f_l(x'))\ dy' + o(1)
\eal\eeq
for $x_n\searrow 0$, since for the second summand we get by using H\"older's inequality
\beq\bal
\Big|\int_{\Sigma\setminus B^{n-1}_\epsilon(x')}Z^{rl}(x-(y',0))f_l(y')\ dy'\Big| &\le x_n\ c \int_{\Sigma\setminus B^{n-1}_\epsilon(x')}\frac{|f_l(y')|}{|x'-y'|^{n}}\ dy'\\
& = x_n c(\epsilon) \in o(1) \mbox{ f�r } x_n\searrow 0.
\eal\eeq
Thus we may conclude that
\beq\bal
|w&_r(x)-f_r(x')|\\
&\le \underbrace{\sup_{y'\in B^{n-1}_\epsilon(x')}|\ff(y')-\ff(x')|}_{\le \delta,\mbox{ if $\epsilon$ is suff. small}}\sum_l \int_{B^{n-1}_\epsilon(x')}|Z^{rl}(x-(y',0))|\ dy' + o(1)\\
&\le c \delta + o(1)
\eal\eeq
for $x_n\searrow 0$ and arbitrary $\delta>0$. It is not hard to see that the constant $c$ is independent of $x_n$. Hence we get
$$\limsup_{x_n\searrow 0} |w_r(x)-f_r(x')| \le c\delta.$$
This finishes the proof that $\bw$ and $\nu$ are smooth solutions.

The estimate \eqref{randabsch}$_1$ now follows from Theorem \ref{hsabschhr} applied to $\bw$ and $\nu_i$,
  $i=1,\ldots, n$. In order to prove \eqref{randabsch}$_2$ we notice that for $1\le k<n$ we have
  \beq\bal
    \pa_k\bw(x)&=\int_\Sigma \bfZ(x'-y',x_n)\pa_k\bfh(y',0)\ dy',\\
    \pa_k \nu(x)&= \pa_i\underbrace{\int_\Sigma z(x'-y',x_n)\pa_kh_i(y',0)\
      dy'}_{=:\nu_{ik}(x)}.  
  \eal\eeq 
  Again Theorem \ref{hsabschhr} applied to $\pa _k \bw$ and $\nu_{ik}$
  gives
  $$
  \norm{\pa_k \nabla\bw}_{L^{\pp}(\setR^n_>)} + \norm{\pa_k
    \nu}_{L^{\pp}(\setR^n_>)}\le c \,\norm{\nabla^2\bfh}_{L^{\pp}(\setR^n_>)}.
  $$ 
  Using the equations \eqref{randsystem}$_{1,2}$ we compute $\pa_n^2
  w_n = - \sum_{i<n}\pa^2_{ni} w_i $, $\pa_n^2 w_i = \pa_i \nu -
  \sum_{j<n} \pa_j^2 w_i$, $1\le i <n$, and $\pa_n \nu= \Delta w_n$. These identities together with the last estimate give the estimate for $\pa _n^2 \bw$ and $\pa _n \nu$. This finishes the proof of the theorem. 
\end{proof}

Using Corollary \ref{stmaingr} and the previous
theorem we get half-space estimates for the Stokes system
\eqref{eq:stokes2}. 

\begin{corollary}\label{stabschhr}
  Let $p \in \PPln(\setRR_>)$ satisfy $1< p^- \le p^+ <\infty$, and let
  $\ff \in (L^{\pp}(\setR^n_>))^n$ and $g \in W^{1,\pp}(\setR^n_>)$
  have bounded support. Let $(\bv,\pi)$ $\in (W^{2,p^-}(\setR^n_>))^n
  \times W^{1,p^-}(\setR^n_>)$ be a solution of the Stokes system in
  the half-space \eqref{eq:stokes2} corresponding to the data $\ff$
  and $g$. Then the first and the second weak derivatives of $\bv$ as well as $\pi$ and its first derivatives belong to the space $L^{\pp}(\setRR_>)$. They satify the estimates
\begin{equation}\bal
  \norm{\nabla \bv}_{L^{\pp}(\setR^n_>)} +
  \norm{\pi}_{L^{\pp}(\setR^n_>)}&\le c 
  \left(\|\ff\|_{D^{-1,\pp}(\setR^n_>)} +
    \norm{g}_{L^{\pp}(\setR^n_>)} \right),
  \\ 
  \norm{\nabla^2 \bv}_{L^{\pp}(\setR^n_>)} + \norm{\nabla
    \pi}_{L^{\pp}(\setR^n_>)}&\le c \left(\|\ff\|_{L^{\pp}(\setR^n_>)}
    + \norm{\nabla g}_{L^{\pp}(\setR^n_>)} \right), 
\label{stabschhrabsch}
\eal\end{equation}
with a constant $c=c(p,n)$.
\end{corollary}

\begin{proof} Note that $\ff\in D^{-1,p(\cdot)}(\setRR_>)$ by Lemma \ref{lem:C0-D-1}. We extend $p$ and $g$ by an even reflection, and $\ff$ by an odd
  reflection to $\setR^n$. Thus $p\in \PPln(\setR^n)$, $\ff \in
  L^{\pp}(\setRR)$ and $g \in W^{1,\pp}(\setR^n)$ with corresponding
  estimates of the whole-space norms by the half-space norms.
  Moreover, $\ff$ has vanishing mean value, and $\ff $ and $g$ still have bounded support. We now construct the whole-space solution $(\widetilde\bv,\widetilde\pi)$ corresponding to this data in the same way as in the proof of Corollary \ref{stmaingr}. Thus we get
  \begin{equation}\bal\label{absch11}
    \norm{\nabla \widetilde\bv}_{L^{\pp}(\setR^n_>)} +
    \norm{\widetilde\pi}_{L^{\pp}(\setR^n_>)}&\le c 
    \left(\|\ff\|_{D^{-1,\pp}(\setR^n_>)} + \norm{g}_{L^{\pp}(\setR^n_>)}
    \right),\\ 
    \norm{\nabla^2 \widetilde\bv}_{L^{\pp}(\setR^n_>)} + \norm{\nabla
      \widetilde\pi}_{L^{\pp}(\setR^n_>)}&\le c 
    \left(\|\ff\|_{L^{\pp}(\setR^n_>)} + \norm{\nabla
        g}_{L^{\pp}(\setR^n_>)} \right),  
    \eal\end{equation}
where we used that $\|\ff\|_{D^{-1,\pp}(\setR^n)}\le \|\ff\|_{D^{-1,\pp}(\setR^n_>)}$. This estimate may be shown in the following way. Every function $\phi$ in $\setRR$ can be split into an even and an odd part:
$$\phi(x) = \frac{\phi(x)+\phi(\tilde{x})}{2} + \frac{\phi(x)-\phi(\tilde{x})}{2} =: \phi_e(x) + \phi_o(x).$$ Since $\ff$ ist odd, we have 
$\int_{\setRR}f_l\phi=\int_{\setRR}f_l\phi_o$.
Hence we get
\beq\bal
\norm{\ff}_{D^{-1,p(\cdot)}(\setRR)} &=  \sup_{\Phi \in \tfr), \atop \norm{\nabla\Phi}_{L^{p'(\cdot)}(\setRR)}\le 1} \int_{\setRR}\ff\cdot\Phi\ dx
=\sup_{\Phi \in \tfr, \atop \norm{\nabla\Phi}_{L^{p'(\cdot)}(\setRR)}\le 1} \int_{\setRR}\ff\cdot\Phi_o\ dx\\
&\le \sup_{\Phi\in\tfr, \textnormal{ odd}, \atop \norm{\nabla\Phi}_{L^{p'(\cdot)}(\setRR)}\le 1} \int_{\setRR}\ff\cdot\Phi\ dx = 2 \sup_{\Phi\in\tfr, \textnormal{ odd}, \atop \norm{\nabla\Phi}_{L^{p'(\cdot)}(\setRR)}\le 1} \int_{\setR^n_>}\ff\cdot\Phi\ dx\\
&\le \sup_{\Phi \in C_0^\infty(\setR^n_>), \atop \norm{\nabla\Phi}_{L^{p'(\cdot)}(\setRR_>)}\le 1} \int_{\setR^n_>}\ff\cdot\Phi\ dx = \norm{\ff}_{D^{-1,\pp}(\setR^n_>)}.
\eal\eeq
The second inequality holds because every odd function $\phi\in\tfr$ is in $D_0^{1,p'(\cdot)}(\setRR_>)$ and $\norm{\nabla\phi}_{p'(\cdot),\setRR}\le 2\norm{\nabla\phi}_{p'(\cdot),\setRR_>}$.

Theorem \ref{randsatz} yields the existence of a solution  $(\bw,\nu)$ of the Stokes system in the half-space \eqref{randsystem} with boundary data  $-\widetilde \bv|_\Sigma$ satisfying the estimates 
  \begin{equation*}\bal
    \norm{\nabla \bw}_{L^{\pp}(\setRR_>)} + \norm{\nu}_{L^{\pp}(\setRR_>)} &\le
    c\,\norm{\nabla \widetilde \bv}_{L^{\pp}(\setRR_>)},\\ 
    \norm{\nabla^2 \bw}_{L^{\pp}(\setRR_>)} + \norm{\nabla \nu}_{L^{\pp}(\setRR_>)} &\le
    c\,\norm{\nabla^2 \widetilde \bv}_{L^{\pp}(\setRR_>)}.
    \eal\end{equation*}
  These estimates together with \eqref{absch11} imply that $\bar \bv := \widetilde \bv +
  \bw $  and $\bar \pi  :=\widetilde \pi +\nu$ satisfy the estimates
  \eqref{stabschhrabsch} and solve the problem
  \beq\bal 
  \Delta \bar {\bv} = \Delta \widetilde \bv + \Delta \bw &= \ff + \nabla
  \bar {\pi} & &\mbox{in } \setR^n_>,\\ 
  \divergenz \bar {\bv} &= g & &\mbox{in } \setR^n_>,\\
  \bar {\bv} = \widetilde \bv -\widetilde \bv &=\mathbf 0 & &\mbox{on } \Sigma.
  \eal\eeq
  If we replace $\pp$ by $p^-$  in the above arguments we get that
  $\bar \bv$ and $\bar \pi$ also satisfy the corresponding estimates
  \eqref{stabschhrabsch} with $\pp$ replaced by $p^-$. Using the classical uniqueness result \cite[Theorem IV.3.3]{Gal94} we deduce that also $\bv $ and $\pi$ satisfy the estimates \eqref{stabschhrabsch}.  
\end{proof}

Now we are ready to prove estimates near the boundary for solutions of the
Stokes system \eqref{eq:stokes} with homogeneous boundary data provided that the boundary is of class $C^{1,1}$.

\begin{definition}
 We say that a domain $\Omega\subset\setRR$ has a $C^{1,1}$-boundary if for every boundary point $\bar{x}\in\partial\Omega$ there is a rotation and translation $G$ of $\setRR$ and a function $a\in C^{1,1}$ such that $G(0)=\bar{x}$, $a(0)=\nabla
  a(0)=0$ and 
  \beq\bal
	\Lambda:=G(\{(x',x_n)\in\setRR | |x'| <\alpha , a(x') = x_n
  \})&\subset \partial \Omega,\\
	V:=G(\{(x',x_n)\in\setRR | |x'|
  <\alpha , a(x') < x_n < a(x') + \beta \})&\subset \Omega,\\
	V_-:=G(\{(x',x_n)\in\setRR | |x'| <\alpha , a(x') -\beta < x_n <
  a(x') \})&\subset \setR^n\setminus \overline\Omega,
  \eal\eeq
for some $\alpha,\beta > 0$.
\end{definition}
Note that in fact the assumption $a(0)=\nabla a(0)=0$ is not a restriction since this simply means that we describe the boundary as the graph of function defined on the tangential space.

\begin{theorem}\label{sttrafohr}
  Let $p \in \PPln(\Omega)$ satisfy $1< p^- \le p^+ <\infty$, and let
  $\ff\in(L^{\pp}(\Omega))^n$ and $g \in W^{1,\pp}(\Omega)$.  Let
  $(\bv,\pi)\in (W^{2,\pp}(\Omega))^n\times W^{1,\pp}(\Omega)$ be a
  solution of the Stokes system \eqref{eq:stokes} with $\bv_0=\mathbf
  0$. Moreover we fix a boundary point $\bar{x}\in\partial\Omega$ and consider the corresponding set $V$ from the previous definition. Moreover define $V'$ analogously to $V$ with $\alpha,\beta$ replaced by $\alpha',\beta'$ where $0 < \alpha ' <\alpha, 0 < \beta ' < \beta$.  If $\alpha$ is chosen sufficiently small
  there exists a constant $c=c(p,V,V',\Omega)$ such that
  $(\bv, \pi)$ satisfy the estimates
\begin{equation}\bal
  &\norm{\nabla \bv}_{L^{\pp}(V')} + \norm{\pi}_{L^{\pp}(V')}\\
  &\hspace{7mm}\le c \big(\|\ff\|_{W^{-1,\pp}(V)} + \norm{g}_{L^{\pp}(V)} +
  \norm{\bv}_{L^{\pp}(V)} + \norm{\pi}_{W^{-1,\pp}(V)} \big),\hspace*{-2mm}\\ 
  &\norm{\nabla^2 \bv}_{L^{\pp}(V')} + \norm{\nabla \pi}_{L^{\pp}(V')}\\
  &\hspace{7mm}\le c \big(\norm{\ff}_{L^{\pp}(V)} + \norm{g}_{W^{1,\pp}(V)} +
  \norm{\bv}_{W^{1,\pp}(V)} +   \norm{\pi}_{L^{\pp}(V)}\big). 
\label{eq:sttrafohr}
\eal\end{equation}
\end{theorem}

\begin{proof}
   For simplicity we assume that the possible rotation and translation
  is not present, i.e.~$G=Id$. We define $V''$ analogously to $V'$ with
  $0<\alpha'<\alpha''<\alpha$ and $ 0<\beta'<\beta''<\beta$. Let
  $\tau\in C^\infty(\overline{\Omega})$ satisfy $\tau = 1$ in $V'$ and
  $\tau = 0$ outside of $V''$. Let us straighten the boundary with the
  help of the coordinate transformation $\bfF: V \rightarrow
  \bfF(V)=:\widehat{V} \subset\setRR$, where
  $(y',y_n):=\bfF(x',x_n):=(x',x_n-a(x'))$ (cf.~Figure \ref{Horst}).

\begin{figure}[h!]
\centering
\input{Horst.pstex_t}
\caption{}     
\label{Horst}
\end{figure}

  We set $\widehat{V}':=\bfF(V')$, $\widehat{\tau}:=\tau\circ
  \bfF^{-1}$, and analogously $\widehat{\bv},
  \widehat{\pi}, \widehat{\ff}$, $\widehat{g}$ and $\widehat{p}$. Note that $\widehat{p}\in \PPln(\widehat{V})$. Furthermore we define $\bar{\bv}
  :=\widehat{\bv} \widehat{\tau} \in
  (W^{2,\widehat{p}(\cdot)}(\setR^n_>))^n $ and $\bar{\pi}:=
  \widehat{\pi}\widehat{\tau} \in
  W^{1,\widehat{p}(\cdot)}(\setR^n_>)$. The integrabilities may be seen by using the identity
\begin{equation}
\norm{\wf}_{L^{\widehat{p}(\cdot)}(\wv)}=\norm{f}_{L^{\pp}(V)}.
\label{normtrafo}
\end{equation}
A tedious but simple computation shows that the couple $(\bar\bv,\bar\pi)$ solves the Stokes system in the half-space \eqref{eq:stokes2} with data $(\bfT,G)$ defined by
  \beq\bal
  T_j  &:= \wh\tau\ \wh f_j  + A_i \pa^2_{in} \bar v_j + B_i\pa_i \wh v_j + C \wh v_j  + D_j \pa_n \bar\pi + E_j \wh\pi,\\
  G &:= \wh\tau\ \wh g + R_i \wh v_i + S_i \pa_n \bar v_i,
  \eal\eeq
  where $B_i,C,E_j,R_i$ are bounded functions depending on $\tau$, $a$ and their derivatives of up to second order. Note that $a\in W^{2,\infty}$. Moreover $A_i,D_j$ and $S_i$ are scalings of first order derivatives of $a$, and thus can be made arbitrarily small by reducing $\alpha$. Hence we will be able to absorb the corresponding terms in the left hand sides of the estimates we are about to derive. Since $(\bfT,G)\in (L^{\widehat{p}(\cdot)}(\setR^n_>))^n\times
  W^{1,\widehat{p}(\cdot)}(\setR^n_>)$ have bounded support
  we can use Corollary \ref{stabschhr} to obtain 
  \beq\bal 
    \norm{\nabla^2\bar{\bv}}_{L^{\wh{p}(\cdot)}(\wv)} &+
    \norm{\nabla\bar{\pi}}_{L^{\wh{p}(\cdot)}(\wv)}\\ 
    & \le c \left(\norm{\wh{\ff}}_{L^{\wh{p}(\cdot)}(\wv)} + 
      \norm{\wh{g}}_{W^{1,\wh{p}(\cdot)}(\wv)} +
      \norm{\wh{\bv}}_{W^{1,\wh{p}(\cdot)}(\wv)} +
      \norm{\wh{\pi}}_{L^{\wh{p}(\cdot)}(\wv)}\right) \\
    & \hspace{1cm} + \frac{1}{2}\left(\norm{\nabla^2\bar{\bv}}_{L^{\bar{p}(\cdot)}(\wv)} + \norm{\nabla\bar{\pi}}_{L^{\wh{p}(\cdot)}(\wv)}\right).
  \eal\eeq
Absorbing into the left hand side we get
\beq\bal 
\norm{\nabla^2\wh{\bv}}_{L^{\wh{p}(\cdot)}(\wv')} &+
    \norm{\nabla\wh{\pi}}_{L^{\wh{p}(\cdot)}(\wv')} \le    
\norm{\nabla^2\bar{\bv}}_{L^{\wh{p}(\cdot)}(\wv)} +
    \norm{\nabla\bar{\pi}}_{L^{\wh{p}(\cdot)}(\wv)}\\ 
    & \le c \left(\norm{\wh{\ff}}_{L^{\wh{p}(\cdot)}(\wv)} + 
      \norm{\wh{g}}_{W^{1,\wh{p}(\cdot)}(\wv)} +
      \norm{\wh{\bv}}_{W^{1,\wh{p}(\cdot)}(\wv)} +
      \norm{\wh{\pi}}_{L^{\wh{p}(\cdot)}(\wv)}\right).
\eal\eeq
  Transforming back via $\bF^{-1}$ we derive the estimate \eqref{eq:sttrafohr}$_2$. Corollary \ref{stabschhr} also gives 
\beq\bal
    \norm{\nabla\bar{\bv}}_{L^{\wh{p}(\cdot)}(\wv)}+ 
    \norm{\bar{\pi}}_{L^{\wh{p}(\cdot)}(\wv)} \le c\, \big (\|\bfT\|_{D^{-1,\wh
      p(\cdot)}(\setRR_>)} + \norm{G}_{L^{\wh{p}(\cdot)}(\wv)}\big ). 
  \eal\eeq
We estimate
\beq\bal
\|\bfT\|&_{D^{-1,\wh
      p(\cdot)}(\setRR_>)}\\
&= \sup_{\bPsi \in C_0^\infty(\setR^n_>), \atop \norm{\nabla\bPsi}_{\wh p'(\cdot),\setR^n_>}\le 1}
\int_{\wv} \Big(\wh\tau\ \wh{\ff}\cdot\bPsi - A_i\pa_i\bar{\bv}\cdot\pa_n \bPsi - \wh{\bv}\cdot \pa_i (B_i\bPsi)\\
&\hspace{3cm} + C \wh{\bv}\cdot\bPsi - \bar{\pi}D_j\pa_n\Psi_j + \wh \pi \Psi_jE_j \Big)\ dx\\
&\le \sup_{\bPsi \in C_0^\infty(\setR^n_>), \atop \norm{\nabla\bPsi}_{\wh p'(\cdot),\setR^n}\le 1} \int_{V} \ff\cdot(\bPsi\circ \bF \tau) + \pi (\Psi_jE_j)\circ \bF\ dx + c\norm{\wh \bv}_{L^{\wh{p}(\cdot)}(\wv)}\\ 
&\hspace{1cm} + \frac{1}{2}\left( \norm{\nabla\bar{\bv}}_{L^{\wh{p}(\cdot)}(\wv)} + \norm{\bar{\pi}}_{L^{\wh{p}(\cdot)}(\wv)}\right)\\
&\le c \left(\norm{\ff}_{W^{-1,\pp}(V)} + \norm{\pi}_{W^{-1,p(\cdot)}(V)} + \norm{\wh \bv}_{L^{\wh{p}(\cdot)}(\wv)}\right) \\
&\hspace{1cm} + \frac{1}{2}\left( \norm{\nabla\bar{\bv}}_{L^{\wh{p}(\cdot)}(\wv)} + \norm{\bar{\pi}}_{L^{\wh{p}(\cdot)}(\wv)}\right)
\eal\eeq
The equality is derived by integrating by parts  the second, third and fifth summand. Note that the functions $A_i$ and $D_j$ do not depend on the last variable. For the first inequality we used H\"older's inequality, estimate \eqref{homest2} and again the fact that $A_i$ and $D_j$ can be made arbitrarily small by reducing $\alpha$. To derive the second inequality note that $\bPsi\circ \bF \tau$ and $(\Psi_jE_j)\circ \bF$ are in $W_0^{1,\pps}(V)$ and that by estimate \eqref{homest2}
$$\norm{\bPsi\circ \bF \tau}_{W^{1,\pps}(V)} + \norm{(\Psi_jE_j)\circ \bF}_{W^{1,\pps}(V)}\le c \norm{\nabla \bPsi}_{L^{\wh p'(\cdot)}(\setR^n_>)}.$$
Furthermore we have
\beq\bal
\norm{G}_{\wh{p}(\cdot),\wv} \le \norm{\wh g}_{\wh{p}(\cdot),\wv} + c\norm{\wh \bv}_{\wh{p}(\cdot),\wv} + \frac{1}{2}\norm{\nabla \bar{\bv}}_{\wh{p}(\cdot),\wv}.
\eal\eeq
Now, proceeding as in the derivation of \eqref{eq:sttrafohr}$_2$ we get \eqref{eq:sttrafohr}$_1$.
\end{proof}

Now we can finally prove the main assertions of this paper.
\begin{proof} (of the Theorems \ref{thm:stokes-strong} and \ref{thm:stokes-weak}) Due to \cite[Theorem IV.6.1]{Gal94} for $\ff\in \tfo$ und $g\in C^\infty(\overline{\Omega})$ with $\int_{\Omega}g=0$ there is a unique strong solution $(\bv,\pi)\in W^{2,\pp}(\Omega)\times W^{1,\pp}(\Omega)$ of the Stokes system \eqref{eq:stokes} with data $\ff$, $g$ and homogenous boundary condition. For each boundary point $\bar{x}\in\partial\Omega$ we may choose sets $V$ and $V'$ like in Theorem \ref{sttrafohr} and analogously defined sets $\Lambda'$ and $V'_-$. Then, the sets $W':=V'\cup\Lambda'\cup V'_-$ form an open cover of the boundary. Since $\partial\Omega$ is compact, we may choose a finite subcover $W_i, i=1,\ldots,m$. Finally we may choose open sets $\Omega_0,\Omega_1\subset\Omega$ such that $\Omega_0 \subset\subset \Omega_1 \subset\subset \Omega$ and $\Omega = \Omega_0 \cup \bigcup_{i=1}^m V_i'$. Now, using the Theorems \ref{sttrafogr} and \ref{sttrafohr} we conclude that
\begin{equation}\bal
&\normpo{\nabla\bv} + \normpo{\pi}\\
&\le \Bignorm{\nabla \bv\Big(\chi_{\Omega_0} + \sum_i \chi_{V_i'}\Big)}_{L^{\pp}(\Omega)} + \Bignorm{\pi \Big(\chi_{\Omega_0} + \sum_i \chi_{V_i'}\Big)}_{L^{\pp}(\Omega)}\\\label{gl}
&\le \norm{\nabla \bv}_{L^{\pp}(\Omega_0)} + \sum_i\norm{\nabla \bv}_{L^\pp(V_i')} + \norm{\pi}_{L^\pp(\Omega_0)} + \sum_i\norm{\pi}_{L^{\pp}(V_i')}\\
&\le c\Big(\|\ff\|_{W^{-1,\pp}(\Omega_1)} + \norm{g}_{L^{\pp}(\Omega_1)} +
  \norm{\bv}_{L^{\pp}(\Omega_1 \setminus \Omega_0)} +
  \norm{\pi}_{W^{-1,\pp}(\Omega_1 \setminus \Omega_0)}\\  &+ \sum_i \|\ff\|_{W^{-1,\pp}(V_i)} + \norm{g}_{L^{\pp}(V_i)} +
  \norm{\bv}_{L^{\pp}(V_i)} + \norm{\pi}_{W^{-1,\pp}(V_i)} \Big)\\
&\le c\Big( \normnegpo{\ff} + \normpo{g} + \normpo{\bv} + \norm{\pi}_{W^{-1,\pp}(V_i)}\Big),
\eal\end{equation}
The last two summands may be eliminated as the subsequent lemma shows. Hence we get the estimate
\begin{equation}\bal
\normepo{\bv} + \normpo{\pi} \le c\left( \norm{\ff}_{W^{-1,\pp}(\Omega)} + \normpo{g}\right).
\label{unglzwei}
\eal\end{equation}
Analogously one shows
\beq\bal
\norm{\nabla^2\bv}_{L^{\pp}(\Omega)} &+ \norm{\nabla\pi}_{L^{\pp}(\Omega)}\\
&\le c \Big(\norm{\ff}_{L^{\pp}(\Omega)} + \norm{g}_{W^{1,\pp}(\Omega)} + \norm{\bv}_{W^{1,\pp}(\Omega)} + \norm{\pi}_{L^{\pp}(\Omega)}\Big).
\eal\eeq
Using \eqref{unglzwei} we conclude that
\begin{equation}\bal
\norm{\bv}_{W^{2,\pp}(\Omega)} + \norm{\pi}_{W^{1,\pp}(\Omega)} &\le c \Big(\norm{\ff}_{L^{\pp}(\Omega)} + \norm{g}_{W^{1,\pp}(\Omega)}\Big).
\label{ungldrei}
\eal\end{equation}
Due to the estimates \eqref{unglzwei} and \eqref{ungldrei} we may continuously extend the linear solution operator to $\lpo\times\wepo$ and $W^{-1,\pp}(\Omega)\times\lpo$, respectively. It is easy to see that these extensions map to strong and weak solutions of the Stokes system \eqref{eq:stokes} with homogenous boundary conditions, respectively. Uniqueness is implied by $\wepo \embedding W^{1,p^-}(\Omega)$ and \cite[Theorem IV.6.1]{Gal94}. Hence the subsequent lemma finishes the proof.
\end{proof}

\begin{lemma} Let $p \in \PPln(\Omega)$ satisfy $1< p^- \le p^+ <\infty$, and let $(\bv,\pi) \in (\wzpo)^n \times \wepo$ a strong solution of the Stokes system \eqref{eq:stokes} with data $\ff\in\lpo$, $g\in\wepo$ and homogenous boundary condition. Then we have the estimate
$$\normpo{\bv} + \normnegpo{\pi} \le c(\normnegpo{\ff} + \normpo{g})$$
with a constant $c=c(p,\Omega)$.
\end{lemma}
\begin{proof} Let us assume that the estimate is wrong. This means that we have a sequence of solutions $(\bv_k,\pi_k)$ of the system with data $\ff_k$ and $g_k$ which satisfy $\normpo{\bv_k} + \normnegpo{\pi_k} = 1$ and $\ff_k \rightarrow 0$ in $W^{-1,\pp}(\Omega)$, $g_k \rightarrow 0$ in $\lpo$.
Since estimate \eqref{gl} holds for $(\bv_k,\pi_k)$ with $\ff,g$ replaced by $\ff_k,g_k$, we conclude, that $(\bv_k,\pi_k)$ is bounded in $W_0^{1,\pp}(\Omega)\times\lpo$. Hence we find subsequences (again denoted by index $k$) satisfying
\beq\bal
\bv_k \weakto \bv \mbox{ in } W_0^{1,\pp}(\Omega),\ \bv_k \rightarrow \bv \mbox{ in } \lpo,\\
\pi_k \weakto \pi \mbox{ in } \lpo,\ \pi_k \rightarrow \pi \mbox{ in } W^{-1,\pp}(\Omega).
\eal\eeq
The strong convergences follow from the compact embeddings $$W_0^{1,\pp}(\Omega)\compactembedding\lpo,$$ and $$\lpo \compactembedding W^{-1,\pp}(\Omega).$$ Thus on one hand we have $\normpo{\bv} + \normnegpo{\pi} = 1$, and on the other hand for all $\bphi\in(\tfo)^n,\ \phi\in\tfo$
\beq\bal
\int_{\Omega} \nabla \bv \cdot \nabla \bphi - \pi \divergenz \bphi\ dx &= \lim_k \int_{\Omega} \nabla \bv_k \cdot \nabla \bphi - \pi_k \divergenz \bphi\ dx 
= \lim_k \int_{\Omega} \ff_k \cdot\bphi\ dx = 0,\\
\int_{\Omega} \divergenz \bv \phi\ dx &= \lim_k \int_{\Omega} \divergenz \bv_k \phi\ dx = \lim_k \int_{\Omega} g_k \phi\ dx = 0.
\eal\eeq
Hence $(\bv,\pi)$ is a weak solution of the system with zero data. From \cite[Theorem IV.6.1]{Gal94} we conclude that $\bv\equiv 0,\ \pi\equiv 0$; a contradiction.
\end{proof}

\section{Poisson problem}
\label{sec:poisson problem}
Let us state in this final section the most important of the analogous results for the \emph{Poisson problem}, cf. \cite{DieHHR10} for a sketch of the proof and \cite{Len08} for full details. Let $\Omega$ be a bounded domain in $\setRR$, $n\ge 2$, with $C^{1,1}-$boundary. Using the techniques\footnote{As we already pointed out it is considerably simpler to obtain the half-space results in the case of the Poisson problem.} we employed in the case of Stokes system one can show that the Poisson problem
\begin{equation}\label{eq:p}
\begin{aligned}
-\Delta u &=f \mbox{ in } \Omega,\\
u &= u_0 \mbox{ on } \partial\Omega,
\end{aligned}
\end{equation}
possesses unique strong and weak solutions, respectively, provided that the data have the appropiate regularity. More precisely, one can prove:
\begin{theorem}\label{thm:p-strong}
  Let $p \in \PPln(\Omega)$ satisfy $1< p^- \le p^+ <\infty$. For arbitrary data $f \in L^{p(\cdot)}(\Omega)$ and $u_0 \in \trace (W^{2,p(\cdot
    )}(\Omega))$ there exists a unique strong solution $u \in \wzpo$ of the Poisson equation \eqref{eq:p} which satisfies the estimate
  \beq\bal
    \normzpo{u} \le c \left(\normpo{f} + \norm{u_0}_{\trace (W^{2,p(\cdot)}(\Omega))} \right), 
  \eal\eeq
  where the constant $c$ depends only on the domain $\Omega$ and the
  exponent $p$.
\end{theorem}

\begin{theorem}
  Let $p \in \PPln(\Omega)$ satisfy $1< p^- \le p^+ <\infty$. For arbitrary data $f \in W^{-1,\pp}(\Omega)$ and $u_0\in \trace(\wepo)$ there exists a unique weak solution $u \in \wepo$ of the Poisson equation \eqref{eq:p} which satisfies the estimate
  \beq\bal
 \normepo{u} \le c'\left( \normnegpo{f} + \norm{u_0}_{\trace(\wepo)}\right),
  \eal\eeq
where the constant $c'$ depends only on the domain $\Omega$ and the
  exponent $p$.
\end{theorem}
We call $u$ a \emph{strong solution} of \eqref{eq:p} provided that it satisfies the differential equation in \eqref{eq:p} in the sense of weak derivatives. We call $u$ a \emph{weak solution} of \eqref{eq:p} provided that
\begin{equation*}
\begin{aligned}
\int_\Omega \nabla u \cdot \nabla\phi\ dx &= \langle f,\phi\rangle & &\forall\phi\in W_0^{1,\pps}(\Omega),\\
u&=u_0 & &\mbox{on } \partial\Omega.
\end{aligned}
\end{equation*}

In order to show the above theorems one needs the following results which are also of interest on their own. Solutions of the equation
\begin{equation}\label{eq:pgr}
\begin{aligned}
-\Delta u =f \mbox{ in } \setR^n
\end{aligned}
\end{equation}
are obtained by a convolution of $f$ with the \emph{Newton potential} 
$$ K(x) := \frac{1}{(n-2)|\partial B_1|}\frac{1}{|x|^{n-2}}.$$
It is well known and easy to see that the second derivatives of the Newton potential satisfy the assumptions of Theorem
\ref{hsabschgr}. Consequently we get:

\begin{lemma}\label{pabl}
  Let $p \in \PPln(\setR^n)$ satisfy $1< p^- \le p^+ <\infty$, and let
  $f\in C^\infty_{0,0}(\setRR)$.  Then the convolution $u:=K\ast f$ is infinitely differentiable and solves the problem \eqref{eq:pgr}. Moreover, the first and second order derivatives have the representations $(i,j=1,\ldots,n)$
  \begin{equation*}\bal
    \partial_i u(x)&= \int_{\setR^n} \partial_{x_i}K(x-y) f(y)\ dy,\\
    \partial_{i}\pa_{j}u(x)&= \lim_{\epsilon \searrow 0}
    \int_{(B(x,\epsilon))^c} \partial_{x_i}\pa_{x_j}K(x-y) f(y)\ dy - \frac{1}{n}\delta_{ij}f(x),
\eal\end{equation*}
and satisfy the estimates 
\begin{equation*}\bal
  \norm{\nabla u}_{L^{\pp}(\setRR)} &\le c \,\|f\|_{D^{-1,\pp}(\setR^n)}\,,
  \\ 
  \normpr{\nabla^2 u}&\le c\, \normpr{f}\,,
\eal\end{equation*}
with a constant  $c=c(p)$.
\end{lemma}

Letting $L$ denote the continuation of the operator $f\mapsto K\ast f$ in the appropriate spaces we can state the following theorem.
\begin{theorem}\label{cor:pabl} 
  Let $p \in \PPln(\setR^n)$ satisfy $1< p^- \le p^+ <\infty$.
  \begin{enumerate}
  \item If $f\in \lpr$ then $L f \in D^{2,p(\cdot)}(\setRR)$
    satisfies the estimate
    \begin{equation*}
      \label{pmaingrabsch1}
      \norm{Lf}_{D^{2,p(\cdot)}(\setRR)} \le c\, \norm{f}_{L^{p(\cdot)}(\setRR)}
    \end{equation*}
    with a constant $c=c(p)$.
  \item If $f\in D^{-1,p(\cdot)}(\setRR)$ then $L f \in
    D^{1,p(\cdot)}(\setRR)$ satisfies the estimate
    \begin{equation*}
      \label{pmaingrabsch2}
      \norm{L f}_{D^{1,p(\cdot)}(\setRR)} \le c\, \norm{f}_{D^{-1,p(\cdot)}(\setRR)}
    \end{equation*}
    with a constant $c=c(p)$.
  \item If $f\in\lpr$ has bounded support and vanishing mean value, hence $f\in D^{-1,\pp}(\setRR)$, then
    $L f \in (D^{(1,2),p(\cdot)}(\setRR))^n\subset(D^{1,p(\cdot)}(\setRR))^n$ satisfies both of the above estimates.
  \end{enumerate}
\end{theorem}

\bibliographystyle{amsalpha}

\bibliography{./lars}

\end{document}